\documentclass[a4paper]{amsart}
\usepackage{graphicx}
\usepackage{amssymb}
\usepackage{amsmath}
\usepackage{amsthm}
\usepackage{amscd}
\usepackage[all,2cell]{xy}

\UseAllTwocells \SilentMatrices
\newtheorem{thm}{Theorem}[section]

\newtheorem{cor}[thm]{Corollary}
\newtheorem{lem}[thm]{Lemma}

\newtheorem{prop}[thm]{Proposition}
\theoremstyle{definition}

\theoremstyle{remark}

\numberwithin{equation}{section}

\begin{document}
\title[Unifying two results of D. Orlov]
{Unifying two results of D. Orlov}
\author[  Xiao-Wu Chen
] {Xiao-Wu Chen}
\thanks{This project was supported by Alexander von Humboldt Stiftung and National Natural Science
Foundation of China (No. 10971206)}
\date{\today}
\thanks{E-mail:
xwchen$\symbol{64}$mail.ustc.edu.cn}

\maketitle

\dedicatory{}%
\commby{}%
\begin{center}
\end{center}

\begin{abstract}
Let $\mathbb{X}$ be a noetherian separated scheme $\mathbb{X}$ of finite Krull dimension
which has enough locally free sheaves of finite rank and let $U\subseteq \mathbb{X}$ be an open subscheme.
We prove that the singularity category of $U$ is triangle equivalent to the Verdier
quotient category of the singularity category of $\mathbb{X}$ with
respect to the thick triangulated subcategory generated by sheaves
supported in the complement of $U$. The result unifies two results
of D. Orlov. We also prove a noncommutative version of this result.
\end{abstract}

\section{Introduction}
Let $\mathbb{X}$ be a noetherian separated scheme of finite Krull dimension which has enough
locally free  sheaves of finite rank. The last assumption means that any coherent sheaf on
$\mathbb{X}$ is a quotient of a locally free sheaf of finite rank (\cite[Chapter III, Ex. 6.4]{Har77}).
Denote by ${\rm coh}\;
\mathbb{X}$ the category of coherent sheaves on $\mathbb{X}$ and by
$\mathbf{D}^b({\rm coh}\;\mathbb{X})$ the bounded derived category.
We will identify ${\rm coh}\; \mathbb{X}$ as the full subcategory of
$\mathbf{D}^b({\rm coh}\;\mathbb{X})$  consisting of stalk complexes
concentrated at degree zero.

Recall that a complex in $\mathbf{D}^b({\rm coh}\; \mathbb{X})$ is \emph{perfect} provided that
locally it is isomorphic to a bounded complex of locally free
sheaves of finite rank (\cite{BGI71}). Under our assumption, perfect complexes are equal to those
isomorphic to a bounded complex of locally free sheaves of finite rank; see \cite[Remark 1.7]{Or04}.
Denote by ${\rm perf}\; \mathbb{X}$ the full triangulated
subcategory of $\mathbf{D}^b({\rm coh}\; \mathbb{X})$ consisting of
perfect complexes; it is a \emph{thick} subcategory, that is, it is
closed under taking direct summands (by \cite[Proposition 1.11]{Or06}).

For the scheme $\mathbb{X}$, D. Orlov introduced in \cite{Or04} a
new invariant, called the \emph{singularity category} of
$\mathbb{X}$, which is defined to be the Verdier quotient
triangulated category $\mathbf{D}_{\rm
sg}(\mathbb{X})=\mathbf{D}^b({\rm coh}\; \mathbb{X})/{\rm perf}\;
\mathbb{X}$. Denote by $q\colon \mathbf{D}^b({\rm coh}\;
\mathbb{X})\rightarrow \mathbf{D}_{\rm sg}(\mathbb{X})$ the quotient
functor. The singularity category $\mathbf{D}_{\rm sg}(\mathbb{X})$
captures certain properties of the singularity of $\mathbb{X}$. For
example, the scheme $\mathbb{X}$ is regular if and only if its
singularity category $\mathbf{D}_{\rm sg}(\mathbb{X})$ is trivial.
Singularity categories are closely related to the Homological Mirror
Symmetry Conjecture; see \cite{Or04, Or06, Or05, Or09}.

We will explain two results of Orlov which we are going to unify.
Let $j\colon U\hookrightarrow \mathbb{X}$ be an open immersion. Note
that the inverse image functor $j^*\colon {\rm coh}\;
\mathbb{X}\rightarrow {\rm coh}\; U$ is exact and then it extends to
a triangle functor $\mathbf{D}^b( {\rm coh}\; \mathbb{X})
\rightarrow \mathbf{D}^b( {\rm coh}\; U)$, which is still denoted by
$j^*$. This functor sends perfect complexes to perfect complexes and
then it induces a triangle functor $\bar{j^*}\colon \mathbf{D}_{\rm
sg}(\mathbb{X}) \rightarrow  \mathbf{D}_{\rm sg}(U)$. Denote by
${\rm Sing}(\mathbb{X})$ the singular locus of $\mathbb{X}$.

The first result is about a local property of singularity
categories; see \cite[Proposition 1.14]{Or04}.

\begin{prop}{\rm (Orlov)}\label{prop:Orlov1}
Use the notation as above. Assume that ${\rm
Sing}(\mathbb{X})\subseteq U$. Then the triangle functor
$\bar{j^*}\colon \mathbf{D}_{\rm sg}(\mathbb{X}) \rightarrow
\mathbf{D}_{\rm sg}(U)$ is an equivalence. \hfill $\square$
\end{prop}

Set $Z=\mathbb{X}\backslash U$ to be the complement of $U$. Denote by
${\rm coh}_Z\; \mathbb{X}$ the full subcategory of ${\rm coh}\;
\mathbb{X}$ consisting of sheaves supported in $Z$. Denote by
$\mathbf{D}^b_Z( {\rm coh}\; \mathbb{X})$ the full triangulated
subcategory of $\mathbf{D}^b( {\rm coh}\; \mathbb{X})$ consisting of
complexes with cohomologies supported in $Z$.

 For a subset
$\mathcal{S}$ of objects in a triangulated category $\mathcal{T}$ we denote by
${\rm thick}\langle \mathcal{S}\rangle$ the smallest thick
triangulated subcategory of $\mathcal{T}$ containing $\mathcal{S}$.
The subcategory ${\rm thick}\langle \mathcal{S}\rangle$ of
$\mathcal{T}$ is said to be \emph{generated} by $\mathcal{S}$. For
example, we have that $\mathbf{D}^b_Z( {\rm coh}\; \mathbb{X})={\rm
thick}\langle {\rm coh}_Z\; \mathbb{X}\rangle$ (by \cite[Chapter I, Lemma
7.2(4)]{Har66}).

Here is the second result of our interest; compare \cite[Proposition
2.7]{Or09}.

\begin{prop}{\rm (Orlov)}\label{prop:Orlov2}
Use the notation as above. Assume that ${\rm
Sing}(\mathbb{X})\subseteq Z$. Then we have $\mathbf{D}_{\rm
sg}(\mathbb{X})={\rm thick}\langle q({\rm coh}_Z\;
\mathbb{X})\rangle$. \hfill $\square$
\end{prop}

The aim of this note is to prove the following result which somehow
unifies the two results of Orlov mentioned above. Let us point out that it is in spirit close
to a result by Krause (\cite[Proposition 6.9]{Kr05}).

\begin{thm}\label{thm:1}
Use the notation as above. Then the triangle functor
$\bar{j^*}\colon \mathbf{D}_{\rm sg}(\mathbb{X}) \rightarrow
\mathbf{D}_{\rm sg}(U)$ induces a triangle equivalence
$$ \mathbf{D}_{\rm sg}(\mathbb{X})/{\rm thick}\langle
q({\rm coh}_Z\; \mathbb{X}) \rangle \; \simeq \; \mathbf{D}_{\rm
sg}(U).$$
\end{thm}

We will see that  Theorem \ref{thm:1} implies Propositions \ref{prop:Orlov1} and \ref{prop:Orlov2} and their
converses; see Corollaries \ref{cor:Cor1} and \ref{cor:Cor2}.

We also prove a noncommutative version of Theorem \ref{thm:1},
which implies  a result in \cite{Ch09} and its converse; see
Theorem \ref{thm:2} and Corollary \ref{cor:Cor3}. Note that for a
left-noetherian $R$ one has the notion of singularity category $\mathbf{D}_{\rm sg}(R)$ of $R$ which is defined analogously  to the singularity
category of a scheme. Let us remark that the singularity category
$\mathbf{D}_{\rm sg}(R)$ was studied by Buchweitz in \cite{Buc87} under
the name ``stable derived category"; also see \cite{Hap91}.

\section{The Proof of Theorem \ref{thm:1}}

In this section we will prove Theorem \ref{thm:1}.

Recall that for a triangle functor $F\colon \mathcal{T}\rightarrow
\mathcal{T}'$ its essential kernel ${\rm Ker}\; F$ is a thick
triangulated category of $\mathcal{T}$ and then $F$ induces uniquely  a
triangle functor $\bar{F} \colon \mathcal{T}/{{\rm Ker}\;
F}\rightarrow \mathcal{T}'$. We say that the functor $F$ is a
\emph{quotient functor} provided that the induced functor $\bar{F}$
is an equivalence.

The following lemma is easy.

\begin{lem}\label{lem:Lemma1}
Let $F\colon \mathcal{T}\rightarrow \mathcal{T}'$ be a quotient
functor. Assume that $\mathcal{N}\subseteq \mathcal{T}$ and
$\mathcal{N}'\subseteq \mathcal{T}'$ are triangulated subcategories
such that $F(\mathcal{N})\subseteq \mathcal{N}'$. Then the induced
functor $\mathcal{T}/\mathcal{N}\rightarrow
\mathcal{T}'/\mathcal{N}'$ by $F$ is a quotient functor.
\end{lem}

\begin{proof}
Denote by $\mathcal{N}''$  the inverse image of $\mathcal{N}'$ under $F$.
Note that ${\rm Ker}\; F\subseteq \mathcal{N}''$ and by assumption
$\mathcal{N}\subseteq \mathcal{N}''$. We identify
$\mathcal{T}$ with $\mathcal{T}/{{\rm Ker}\; F}$ via $\bar{F}$. Then we have
$\mathcal{T}'/\mathcal{N}'\simeq \mathcal{T}/\mathcal{N}''\simeq
(\mathcal{T}/\mathcal{N}')/(\mathcal{N}''/\mathcal{N'})$; see
\cite[Chapitre 1, \S2, 4-3 Corollaire]{Ver77}. This proves that the
induced functor $\mathcal{T}/\mathcal{N}\rightarrow
\mathcal{T}'/\mathcal{N}'$ is a quotient functor.
\end{proof}

Let $\mathbb{X}$ be the scheme as in Section 1. Let $j\colon
U\hookrightarrow \mathbb{X}$ be an open immersion and let
$Z=\mathbb{X}\backslash U$ be the complement of $U$. The following lemma is
well known; see \cite[Lemma 2.2]{Or09}.

\begin{lem}\label{lem:Lemma2}
The inverse image functor $j^*\colon \mathbf{D}^b( {\rm coh}\;
\mathbb{X}) \rightarrow \mathbf{D}^b( {\rm coh}\; U)$ is a quotient
functor with ${\rm Ker}\; j^*=\mathbf{D}^b_Z({\rm coh}\;
\mathbb{X})$.
\end{lem}

\begin{proof}
It is direct to see that ${\rm Ker}\; j^*=\mathbf{D}^b_Z({\rm coh}\;
\mathbb{X})$. Note that the inverse image functor $j^*\colon {\rm
coh}\; \mathbb{X}\rightarrow {\rm coh}\; U$ is exact and its
essential kernel is ${\rm coh}_Z\; \mathbb{X}$. In particular, ${\rm
coh}_Z\; \mathbb{X}\subseteq {\rm coh}\; \mathbb{X}$ is a Serre
subcategory. The functor $j^*$ induces an equivalence ${\rm coh}\;
\mathbb{X}/{\rm coh}_Z\; \mathbb{X}\simeq {\rm coh}\; U$; compare \cite[Chapter VI, Proposition 3]{Ga62}.
Here ${\rm
coh}\; \mathbb{X}/{\rm coh}_Z\; \mathbb{X}$ is the quotient abelian
category in the sense of \cite{Ga62}. Now the result follows
immediately from \cite[Theorem 3.2]{Miy91}.
\end{proof}

\vskip 5pt

 \noindent {\bf Proof of Theorem \ref{thm:1}:}\quad  By
combining Lemmas \ref{lem:Lemma1} and \ref{lem:Lemma2} we infer that
the functor $\bar{j^*}\colon \mathbf{D}_{\rm sg}(\mathbb{X})
\rightarrow \mathbf{D}_{\rm sg}(U)$ is a quotient functor. Then it
suffices to show that ${\rm Ker}\; \bar{j^*}={\rm thick}\langle
q({\rm coh}_Z\; \mathbb{X}) \rangle$. Since $j^*$ vanishes on ${\rm
coh}_Z\; \mathbb{X}$, we have ${\rm thick}\langle q({\rm coh}_Z\;
\mathbb{X}) \rangle\subseteq {\rm Ker}\; \bar{j^*}$.

 To show ${\rm Ker}\; \bar{j^*} \subseteq {\rm thick}\langle q({\rm coh}_Z\; \mathbb{X})
\rangle$, let $\mathcal{F}^\bullet$ be a bounded complex of sheaves
lying in ${\rm Ker}\; \bar{j^*}$, that is, the complex
$j^*(\mathcal{F}^\bullet)$ is perfect. We will show that
$\mathcal{F}^\bullet \in {\rm thick}\langle q({\rm coh}_Z\;
\mathbb{X}) \rangle$. Let us remark that the argument below is the same
as the one in the proof of \cite[Proposition 2.7]{Or09}. Since
$\mathbb{X}$ has enough locally free sheaves of finite rank, there
exists a bounded complex $\cdots \rightarrow 0\rightarrow
\mathcal{K}^{-n}\rightarrow \mathcal{E}^{1-n}\rightarrow
\mathcal{E}^{2-n}\rightarrow \cdots \rightarrow
\mathcal{E}^m\rightarrow \cdots $ with each $\mathcal{E}^j$ locally
free which is quasi-isomorphic to $\mathcal{F}^\bullet$; moreover,
we may take $n$ arbitrarily large; see \cite[Chapter I, Lemma
4.6(1)]{Har66}. By a brutal truncation we obtain a morphism
$\alpha\colon \mathcal{F}^\bullet\rightarrow \mathcal{K}^{-n}[n]$ in
$\mathbf{D}^b({\rm coh}\; \mathbb{X})$ whose cone is a perfect
complex. Hence $q(\alpha)$ is an isomorphism. On the other hand,
since $j^*(\mathcal{F}^\bullet)$ is perfect, taking $n$ large enough
we get that $j^*(\alpha)$ is zero (in $\mathbf{D}^b({\rm coh}\;
U)$); see \cite[Proposition 1.11]{Or06}. By Lemma \ref{lem:Lemma2} this implies that $\alpha$ factors
through a complex in $\mathbf{D}^b_Z({\rm coh}\; \mathbb{X})$ and
then the isomorphism $q(\alpha)$ factors through an object in
$q(\mathbf{D}^b_Z({\rm coh}\; \mathbb{X}))$. This implies that
$\mathcal{F}^\bullet$ is a direct summand of an object in
$q(\mathbf{D}^b_Z({\rm coh}\; \mathbb{X}))$. Recall that
$\mathbf{D}^b_Z( {\rm coh}\; \mathbb{X})={\rm thick}\langle {\rm
coh}_Z\; \mathbb{X}\rangle$. This implies that $\mathcal{F}^\bullet
\in {\rm thick}\langle q({\rm coh}_Z\; \mathbb{X}) \rangle$. \hfill
$\square$

\vskip 5pt

Theorem \ref{thm:1} implies  Proposition
\ref{prop:Orlov1} and its converse. Recall that ${\rm Sing}(\mathbb{X})$ denotes the
singular locus of $\mathbb{X}$.

\begin{cor}\label{cor:Cor1}
Use the notation as above. Then $\bar{j^*}\colon  \mathbf{D}_{\rm
sg}(\mathbb{X}) \rightarrow \mathbf{D}_{\rm sg}(U)$ is an
equivalence if and only if ${\rm Sing}(\mathbb{X})\subseteq U$.
\end{cor}

\begin{proof}
By Theorem \ref{thm:1} it suffices to show that ${\rm
coh}_Z(X)\subseteq {\rm perf}\; \mathbb{X}$ if and only if ${\rm
Sing}(\mathbb{X})\subseteq U$. Recall that a coherent sheaf
$\mathcal{F}$ on $\mathbb{X}$ is perfect if and only if at each  point $x\in
X$, the stalk $\mathcal{F}_x$, as an $\mathcal{O}_{X,x}$-module,  has finite
projective dimension; see \cite[Chapter III, Ex. 6.5]{Har77}.
Then the ``if" part follows. To see the ``only if" part, assume on the
contrary that ${\rm Sing}(\mathbb{X})\nsubseteq U$. Then there is a singular closed
point $x\in Z$. The skyscraper sheaf $k(x)\in {\rm coh}_Z\;
\mathbb{X}$ is not perfect. A contradiction!
\end{proof}

Theorem \ref{thm:1} implies Proposition \ref{prop:Orlov2} and its converse.

\begin{cor}\label{cor:Cor2}
Use the notation as above. Then $\mathbf{D}_{\rm
sg}(\mathbb{X})={\rm thick}\langle q({\rm coh}_Z\;
\mathbb{X})\rangle$ if and only if ${\rm Sing}(\mathbb{X})\subseteq
Z$.
\end{cor}

\begin{proof}
Recall that $\mathbf{D}_{\rm sg}(U)$ is trivial if and only if $U$
is regular, or equivalently, ${\rm Sing}(\mathbb{X})\subseteq Z$.
Then we are done by Theorem \ref{thm:1}.
\end{proof}

We would like to point out that Corollary \ref{cor:Cor2} generalizes slightly a result
by Keller, Murfet and Van den Bergh (\cite[Proposition A.2]{KMV09}). Let $R$ be
a commutative noetherian local ring with its maximal ideal
$\mathfrak{m}$ and its residue field $k=R/\mathfrak{m}$.  Set $\mathbb{X}={\rm Spec}(R)$ and
$U=\mathbb{X}\backslash \{\mathfrak{m}\}$. We identify ${\rm coh}\; \mathbb{X}$ with
the category of finitely generated $R$-modules, and then ${\rm coh}_{\{\mathfrak{m}\}} \; \mathbb{X}$ is identified
with the category of finite length $R$-modules. We infer from Corollary \ref{cor:Cor2} that
${\rm Sing}(\mathbb{X})\subseteq \{\mathfrak{m}\}$, that is, the ring $R$ is
an \emph{isolated singularity} if and only if $\mathbf{D}_{\rm
sg}(\mathbb{X})={\rm thick}\langle q(k)\rangle$, that is, the singularity category
$\mathbf{D}_{\rm sg}(\mathbb{X})$ is generated by the residue field
$k$.

\section{A noncommutative version of Theorem \ref{thm:1}}

In this section we will prove a noncommutative version of Theorem
\ref{thm:1}.

Let $R$ be a left-noetherian ring. Denote by $R\mbox{-mod}$ the
category of finitely generated left $R$-modules and by
$\mathbf{D}^b(R\mbox{-mod})$ the bounded derived category. We denote
by $R\mbox{-proj}$ the full subcategory of $R\mbox{-mod}$ consisting
of projective modules and by $\mathbf{K}^b(R\mbox{-proj})$ the
bounded homotopy category. We may view $\mathbf{K}^b(R\mbox{-proj})$
as a thick triangulated subcategory of $\mathbf{D}^b(R\mbox{-mod})$ (compare \cite[1.2]{Buc87} and \cite[1.4]{Hap91}).
Following Orlov (\cite{Or05}), the \emph{singularity category} of the ring $R$ is
defined to be the Verdier quotient triangulated category
$\mathbf{D}_{\rm
sg}(R)=\mathbf{D}^b(R\mbox{-mod})/\mathbf{K}^b(R\mbox{-proj})$; compare \cite{Buc87, Hap91}. We
denote by $q\colon \mathbf{D}^b(R\mbox{-mod})\rightarrow
\mathbf{D}_{\rm sg}(R)$ the quotient functor.

Let $e\in R$ be an idempotent. It is direct to see that $eRe$ is a
left-noetherian ring. Note that $eR$ has a natural
$eRe$-$R$-bimodule structure inherited from the multiplication of
$R$. Consider the Schur functor
$$S_e=eR\otimes_R- \colon R\mbox{-mod}\longrightarrow eRe\mbox{-mod}.$$
It is an exact functor. We denote by $\mathcal{B}_e$ the essential
kernel of $S_e$. Note that an $R$-module $M$ lies in $\mathcal{B}_e$
if and only if $eM=0$, and if and only if $(1-e)M=M$. The Schur
functor $S_e$ extends to a triangle functor
$\mathbf{D}^b(R\mbox{-mod})\rightarrow
\mathbf{D}^b(eRe\mbox{-mod})$, which is still denoted by $S_e$.

We assume that the left $eRe$-module $_{eRe}eR$ has finite
projective dimension. In this case, the functor $S_e\colon \mathbf{D}^b(R\mbox{-mod})
\rightarrow \mathbf{D}^b(eRe\mbox{-mod})$ sends
$\mathbf{K}^b(R\mbox{-proj})$ to $\mathbf{K}^b(eRe\mbox{-proj})$.
Then it induces uniquely a triangle functor $\bar{S_e}\colon \mathbf{D}_{\rm
sg}(R)\rightarrow \mathbf{D}_{\rm sg}(eRe)$. The question on when
the functor $\bar{S_e}$ is an equivalence  was studied
in \cite{Ch09}.

The following result can be viewed as a (possibly naive) noncommutative version of
Theorem \ref{thm:1}.

\begin{thm}\label{thm:2}
Let $R$ be  a left-noetherian ring and let $e\in R$ be an
idempotent. Assume that the left $eRe$-module $_{eRe}eR$ has finite
projective dimension. Then the triangle functor $\bar{S_e}\colon
\mathbf{D}_{\rm sg}(R)\rightarrow \mathbf{D}_{\rm sg}(eRe)$ induces
a triangle equivalence
$$\mathbf{D}_{\rm sg}(R)/{\rm thick}\langle q(\mathcal{B}_e)\rangle\; \simeq \; \mathbf{D}_{\rm sg}(eRe).$$
\end{thm}

Denote by $\mathcal{N}_e$ the full triangulated subcategory of
$\mathbf{D}^b(R\mbox{-mod})$ consisting of complexes with
cohomologies in $\mathcal{B}_e$. Note that $\mathcal{N}_e={\rm thick}\langle
\mathcal{B}_e\rangle$. We will use the following
well-known lemma.

\begin{lem}{\rm (\cite[Lemma 2.2]{Ch09})}\label{lem:Lemma3}
The triangle functor $S_e\colon \mathbf{D}^b(R\mbox{-{\rm
mod}})\rightarrow \mathbf{D}^b(eRe\mbox{-{\rm mod}})$ is a quotient
functor with ${\rm Ker}\; S_e=\mathcal{N}_e$.  \hfill $\square$
\end{lem}

\vskip5pt

\noindent{\bf Proof of Theorem \ref{thm:2}:}\quad By combining
Lemmas \ref{lem:Lemma1} and \ref{lem:Lemma3} we infer that the
functor  $\bar{S_e}\colon \mathbf{D}_{\rm sg}(R)\rightarrow
\mathbf{D}_{\rm sg}(eRe)$ is a quotient functor.  Then it suffices
to show that ${\rm Ker}\; \bar{S_e}={\rm thick}\langle
q(\mathcal{B}_e)\rangle$. Then the same proof with the one of
Theorem \ref{thm:1} will work. We simply replace locally free sheaves of
finite rank by projective modules. We omit the details here. \hfill
$\square$

\vskip 5pt

We need some notions. An idempotent $e\in R$ is said to be
\emph{regular} provided that every module in $\mathcal{B}_{1-e}$ has
finite projective dimension; $e$ is said to be
\emph{singularly-complete} provided that the idempotent $1-e$ is
regular; see \cite{Ch09}.

We have the following immediate consequence of Theorem \ref{thm:2},
which contains \cite[Theorem 2.1]{Ch09} and its converse. Let us point out that this result
is quite useful to describe the singularity categories
for a certain class of rings; see \cite{Ch09}.

\begin{cor}\label{cor:Cor3}
Keep the assumption as above. Then  $\bar{S_e}\colon \mathbf{D}_{\rm
sg}(R)\rightarrow \mathbf{D}_{\rm sg}(eRe)$ is an equivalence if and
only if the idempotent $e$ is singularly-complete. \hfill $\square$
\end{cor}

\bibliography{}

\vskip 10pt

 {\footnotesize \noindent Xiao-Wu Chen, Department of
Mathematics, University of Science and Technology of
China, Hefei 230026, P. R. China \\
Homepage: http://mail.ustc.edu.cn/$^\sim$xwchen \\
\emph{Current
address}: Institut fuer Mathematik, Universitaet Paderborn, 33095,
Paderborn, Germany}

\end{document}